\documentclass[12pt,reqno, a4paper]{amsart}

\usepackage{amsmath}
\usepackage{amsfonts}
\usepackage{amssymb}
\usepackage{mathtools}
\usepackage{tikz-cd}

\usepackage{amsthm}
\usepackage{float}
\usepackage{geometry}
\usepackage{rotating}
\usepackage{multirow}

\usepackage[ color = blue!20]{todonotes}
\usepackage{hyperref}
\usepackage[capitalise,noabbrev]{cleveref}
\creflabelformat{equation}{#2\textup{#1}#3} 
\usepackage{mathtools}
\usepackage{graphicx}
\graphicspath{ {./images/} }

\usepackage[OT2,T1]{fontenc}

\DeclareSymbolFont{cyrletters}{OT2}{wncyr}{m}{n}
\DeclareMathSymbol{\Sha}{\mathalpha}{cyrletters}{"58}

\theoremstyle{definition}
\newtheorem{theorem}{Theorem}[section]
\numberwithin{equation}{theorem}
\newtheorem{definition}[theorem]{Definition}
\newtheorem{proposition}[theorem]{Proposition}

\newtheorem{remark}[theorem]{Remark}

\newcommand{\Z}{\mathbb{Z}}
\newcommand{\C}{\mathbb{C}}

\newcommand{\Q}{\mathbb{Q}}
\newcommand{\F}{\mathbb{F}}

\newcommand{\abs}[1]{\lvert #1 \rvert} 

\newcommand{\SL}{\mathrm{SL}}

\newcommand{\GL}{\mathrm{GL}}
\newcommand{\mf}[1]{\mathfrak{#1}}

\newcommand{\p}{\mf{p}}
\newcommand{\q}{\mf{q}}

\newcommand{\Mod}[1]{\ (\mathrm{mod}\ #1)}

\newcommand{\demph}[1]{\textbf{#1}}
\renewcommand{\L}{\mathcal{L}}

\newcommand{\Gal}{\mathrm{Gal}}

\newcommand{\Kbar}{\overline{K}}
\newcommand{\im}{\mathrm{im}}
\newcommand{\unr}{\mathrm{unr}}
\newcommand{\Sel}{\mathrm{Sel}}
\newcommand{\no}{\mathrm{NO}}
\newcommand{\rel}{\mathrm{rel}}
\newcommand{\bk}{\mathrm{BK}}
\newcommand{\Hom}{\mathrm{Hom}}
\newcommand{\SD}{\mathrm{SD}}
\newcommand{\rank}{\mathrm{rank}}

\newcommand{\OO}{\mathcal{O}}

\newcommand{\LMRefDB}[2]{\href{#1}{\texttt{#2}}}
\newcommand{\spc}{\hspace{1pc}}

\title[Computing mod $p$ Selmer groups]{Computing Selmer groups associated to mod $\boldsymbol{p}$ Galois representations}
\author{Lewis Combes}
\date{\today}
\begin{document}
	
	\begin{abstract}
		We present methods to compute Selmer groups associated to mod $p$ Galois representations $\rho$ over a number field $K$, with a particular focus on comparing their ranks with periods coming from cohomology classes associated to $\rho$ by Serre's conjecture. This provides evidence for a loose version of a ``mod $p$ Bloch-Kato conjecture'', where the vanishing of a period is predicted to capture the presence of rank in a Selmer group. Our methods are explicit, and implemented in Magma. 
	\end{abstract}
	
	\maketitle
	\tableofcontents
	
	\section{Introduction}
	
	\subsection{Motivation}
	Selmer groups are objects of fundamental importance in modern number theory. They are associated to Galois representations, and conjectured to contain important arithmetic information. Robust methods for computing with Selmer groups are useful in formulating and verifying these conjectures. 
	
	We are motivated by the question of Calegari and Venkatesh, posed in Section 10.3 of their book \cite{CalegariVenkatesh}: 
	\begin{quote}
		Do periods of torsion classes detect classes in Galois cohomology?
	\end{quote}
	This question asks, loosely, if there is a relationship between the rank of a Selmer group and a \textit{period} associated to a class in the mod $p$ cohomology of an arithmetic group. This relationship should be of the form 
	\begin{equation*}
		\text{period} = 0 \iff \text{rank} > 0.
	\end{equation*}
	The purpose of this paper is to develop methods to compute the ranks of various mod $p$ Selmer groups, namely the \textit{relaxed}, \textit{nearly-ordinary} and \textit{unramified} groups (see \cref{SectionGaloisCohomology}), and use these to test the question of Calegari-Venkatesh. Our main result is the computation presented in \cref{SectionBKTest}, that a weaker version of this relationship may hold for the nearly-ordinary Selmer group. 
	
	The relationship between periods and ranks is a direct generalisation of the Bloch-Kato conjecture for $p$-adic Galois representations $\rho$, which associates to $\rho$ an $L$-function $L(\rho,s)$ and a Selmer group $\Sel_{\mathrm{BK}}(\rho)$, and predicts the relationship 
	\begin{equation}\label{EquationBlochKato}
		L(\rho,1) = 0 \iff \rank(\Sel_{\bk}(\rho)) > 0.
	\end{equation}
	In fact, the relationship has some extra subtleties (see Bloch-Kato \cite{BlochKato} for further details), but this is the central principle we wish to test. 
	
	There is not currently a robust choice for a mod $p$ equivalent of the Bloch-Kato Selmer group, and so we use the nearly-ordinary group to formulate a more oblique version of this relationship in the mod $p$ setting. In the $p$-adic setting, the nearly-ordinary Selmer group over-estimates the rank Bloch-Kato Selmer group by at most $1$; per our computations, the same may hold in this case, with periods being used to predict the rank of a hypothetical ``mod $p$ Bloch-Kato Selmer group'' that is also overestimated by at most $1$. 
	
	We have developed code in \textsf{Magma} to perform all our computations, which is available in an associated GitHub repository \cite{CombesGitHub}. It is our hope that it will be useful to others for further refining these conjectures. During the preparation of this work, a paper by Etienne \cite{Etienne} was released, also describing a method to compute Selmer groups associated to mod $p$ Galois representations, although with a different approach.

	\subsection{Structure of the paper}
	In \cref{SectionPreliminaries}, we introduce standard concepts from the representation theory of infinite Galois groups. In \cref{SectionGaloisCohomology,SectionCohoToCFT}, we define Selmer groups for mod $p$ representations and connect them to abelian number field extensions. This provides the backbone of our approach to computing Selmer groups, using class field theory. In \cref{SectionTechnical} we note some technical aspects associated to the computations, including proving that the ranks we compute are independent of the choices required along the way. We illustrate the method with examples in \cref{SectionExamples}, and give some statistics on computations with these groups. In \cref{SectionBKTest}, we give computational evidence in favour of a relationship between the nearly-ordinary Selmer group and a hypothetical ``mod $p$ version'' of the Bloch-Kato Selmer group. 
	
	Our primary source of Galois representations will be as the $p$-torsion points of elliptic curves, data for which we obtain from the LMFDB \cite{LMFDB}. We make frequent use of code by Sutherland \cite{SutherlandTorsion} to compute the number fields giving these representations.

	\subsection{Acknowledgments}
	The bulk of this work was completed during the author's Ph.D. studies at the University of Sheffield, while in receipt of EPSRC grant EP/R513313/1. Much of the content of this paper appears in the author's thesis, \textit{Periods and Selmer groups associated to mod $p$ Galois representations over imaginary quadratic fields}, albeit with different words in different orders. We would like to thank Haluk \c{S}eng\"{u}n for suggesting and supervising this Ph.D. project, and for all his guidance and support during its completion. The remainder of this work was done at the University of Sydney, under the Australian Research Council Laureate Fellowship FL230100256 grant of Geordie Williamson. We are grateful to acknowledge this funding, and the support of the Sydney Mathematics Research Institute. We are also grateful to John Voight for helpful conversations and comments, and H\aa vard Damm-Johnsen for comments on an earlier draft.

	\section{Preliminaries and notation}\label{SectionPreliminaries}

	In this section we list the basic constructions we will use throughout. The notational choices we make here carry over to the rest of the paper, unless otherwise specified. That is, $K$ will always denote a number field, $\rho$ a mod $p$ Galois representation, and so on, unless we need these symbols for something else, which will be stated. All constructions of infinite Galois theory we use are standard, and can be found in many textbooks, e.g. ~Chapter 2 of Koch \cite{KochGalois}. 	
	
	Let $K$ be a number field. We fix an algebraic closure $\Kbar$ of $K$, and write $G_K = \Gal(\Kbar/K)$ for the absolute Galois group of $K$. The standard construction of $\Gal(\Kbar/K)$ is via an inverse limit: for two Galois extensions $F_1/K$, $F_2/K$ with $F_2 \subset F_1$, we have a map $\Gal(F_1/K) \to \Gal(F_2/K)$ given by restriction; this forms an inverse system, the limit of which is $\Gal(\Kbar/K)$. 	
	
	Let $\rho \colon G_K \to \GL(V)$ be a finite-dimensional mod $p$ Galois representation (so the vector space $V$ is isomorphic to $\F_{q}^n$ for some finite field $\F_q$ of characteristic $p$). Further, assume that $\rho$ is a continuous representation, with respect to the Krull topology on $G_K$ and the discrete topology on $V$. Then $\im(\rho)$ is a finite group, and $\rho$ factors through a finite Galois extension of number fields $L/K$: 
	\begin{equation*}
		\im(\rho) \simeq G_K/\ker(\rho) \simeq \Gal(L/K), \spc L = \Kbar^{\ker(\rho)}.
	\end{equation*}
	
	Note that the Galois correspondence for $G_K$ tells us that $\ker(\rho) = \Gal(\overline{L}/L)$, which we will denote by $G_L$.

	We will write $\p$ to denote a place of $K$, not necessarily finite and not necessarily over the characteristic $p$ of $V$. For each $\p$, we fix a place $\mathfrak{P}$ of $\Kbar$; in effect, this makes a choice of place lying over $\p$ for each finite Galois extension $F/K$ by \textit{restriction}: take a representative absolute value in the equivalence class $\mf{P}$ and restrict it to the subfield $F$; its equivalence class is a place of $F$, denoted $\mf{P}\vert_{F}$. 
	
	The choice of $\mf{P}$ is \textit{compatible} with $G_K$ in the following sense: for a tower $F_1/F_2/K$ of finite extensions, with $F_1/K$ and $F_2/K$ Galois, we have $(\mf{P}\vert_{F_1}) \vert_{F_2} = \mf{P} \vert_{F_2}$.

	We denote by $\q$ the choice of place this fixes in $L$, i.e. $\q = \mf{P}\vert_{L}$. For each field $F = K,L,\overline{K}$ and its respective place $I = \p, \q, \mf{P}$, we have the following: the completion $F_I$, its ring of integers $\OO_{F_I}$, the maximal ideal $\mf{m}_{I}$ of this ring, and its residue field $\F_{I} = \OO_{F_I} / \mf{m}_I$. Then we have $D_\p = \Gal(\Kbar_{\mf{P}}/K_{\p})$, the \demph{decomposition group at $\p$}. When $\p$ is finite, we have the \demph{inertia subgroup} 
	\begin{equation*}
		I_{\p} = \ker( \Gal(\Kbar_{\mf{P}} / K_{\p}) \to \Gal(\F_{\mf{P}} / \F_{\p}) ) \leq D_{\p}. 
	\end{equation*}
	When $\p$ is infinite, one defines the inertia subgroup $I_{\p} = D_{\p}$. 	
	
	We also have an inertia subgroup associated to $\q$, which we will use in \cref{SectionCohoToCFT}. It is given by 
	\begin{equation*}
		I_{\q} = \ker( \Gal(\Kbar_{\mf{P}} / L_{\q}) \to \Gal(\F_{\mf{P}} / \F_{\q}) ).
	\end{equation*}
	
	The decomposition and inertia subgroups of $G_K$ have their realisations in the finite Galois group $\Gal(L/K) \simeq G_K/G_L$, which are the standard groups from finite Galois theory: 
	\begin{equation*}
		D_{\p}(L/K) = \{ \sigma \in \Gal(L/K) \mid \sigma(\q) = \q \},
	\end{equation*}
	\begin{equation*}
		I_{\p}(L/K) = \{ \sigma \in \Gal(L/K) \mid \sigma(x) - x \in \q ~ \text{ for all } x \in \OO_L \}.
	\end{equation*}
	We can also write 
	\begin{equation}\label{EquationInfiniteToFiniteDecomp}
		D_{\p}(L/K) \simeq \rho(D_{\p}) \simeq D_{\p}/(D_{\p} \cap G_L),
	\end{equation}
	\begin{equation}\label{EquationInfiniteToFiniteInertia}
		 I_{\p}(L/K) \simeq \rho(I_{\p}) \simeq I_{\p}/(I_{\p} \cap G_L)
	\end{equation}
	by the first isomorphism theorem.

	The groups $D_{\p}$, $I_{\p}$ (and their finite realisations) are all only defined up to the choice of place $\mf{P}$ in $\Kbar$ (for the finite groups, the place $\q$ of $L$). We will see later that working with respect to a chosen fixed place in $L$ gives ranks that are independent of this choice. In the course of proving this, we will need to refer to the decomposition or inertia groups given by a \textit{specific choice} of place $\q$ in $L$; in these instances, we will write them as $D_{\p}^{\q}(L/K)$ and $I_{\p}^{\q}(L/K)$.

	Finally, we will also make extensive use of the following property of some two-dimensional representations. 
	
	\begin{definition}
		A representation $\rho \colon G_K \to \GL_2(\F)$ is called \demph{nearly-ordinary at $\p$} if $\rho\vert_{D_{\p}}$ fixes a one-dimensional subset $\ell$ (a line) in $V$ (as a set). That is, there is an exact sequence 
		\begin{equation*}
			0 \to \ell \to V \to V/\ell \to 0
		\end{equation*}
		of $D_{\p}$-modules of dimensions $1$, $2$ and $1$.
	\end{definition}
	
	\begin{remark}
		The terms ``ordinary'' and ``nearly-ordinary'' are not standard across the literature, with both referring to Galois representations having a filtration with varying extra technical conditions. In particular, some authors require nearly-ordinary representations to satisfy that the action of the quotient representation $\rho^*$ on $V/\ell$ is unramified (i.e. trivial on inertia). We make no such imposition, only requiring the existence of the filtration. However, we will examine the difference between the ramified and unramified quotients of a nearly-ordinary representation in \cref{SectionBKTest}. 
	\end{remark}
	
	Our interest in the nearly-ordinary Selmer group comes from its connection to the Bloch-Kato Selmer group in the $p$-adic case. For $\rho$ the $p$-adic Galois representation associated to an elliptic curve $E$ satisfying 
	\begin{enumerate}
		\item $\rho$ is nearly-ordinary (in our sense),
		\item the quotient representation $\rho^*$ is unramified at $p$,
	\end{enumerate}
	one can associate the Greenberg Selmer group $\Sel_{\mathrm{Gr}}(\rho)$ (for us, the nearly-ordinary Selmer group associated to the unramified quotient). One then has 
	\begin{equation*}
		\rank(\Sel_{\mathrm{Gr}}(\rho)) = \rank(\Sel_{\mathrm{BK}}(\rho)) + \begin{cases}
			1 & L(\rho,s) \text{ has a trivial zero} \\ 
			0 & \text{else}.
		\end{cases}
	\end{equation*}
	That is, the Greenberg Selmer group overestimates the rank of the Bloch-Kato Selmer group by at most $1$.

	\section{Galois cohomology and Selmer groups}\label{SectionGaloisCohomology}

	\begin{definition}
		A \demph{local condition at $\boldsymbol{\p}$} is a choice of subspace $L_{\p} \leq H^1(D_{\p},V)$. The \demph{unramified condition} is the subspace 
		\begin{equation*}
			H^1_{\unr}(D_{\p},V) \coloneqq \ker\left( H^1(D_{\p},V) \to H^1(I_{\p},V) \right).
		\end{equation*}
		A \demph{Selmer system} $\L = \{ L_{\p} \mid \p \text{ a place of } K \}$ is a choice of local condition for each place of $K$, such that $L_{\p} = H^1_{\unr}(D_{\p},V)$ for all but finitely many $\p$. 
	\end{definition}
	
	\begin{definition}
		The \demph{Selmer group} associated to a given Selmer system $\L$ is the group 
		\begin{equation*}
			\Sel_{\L}(\rho) = \ker\left( H^1(G_K,V) \to \prod_{\p} \frac{H^1(D_{\p},V)}{L_{\p}} \right).
		\end{equation*}
		That is, the Selmer group is all elements of $H^1(G_K,V)$ that  simultaneously satisfy all local conditions. 
	\end{definition}
	
	\begin{definition}
		The \demph{Bloch-Kato Selmer group} associated to a $p$-adic Galois representation $V$ is defined by the Selmer system 
		\begin{equation*}
			L_{\p} = \begin{cases}
			H^1_{\unr}(D_\p,V) & \p \nmid p \\ 
			\ker(H^1(D_\p,V) \to H^1(D_\p,V \otimes_{\Q_p} B_{\mathrm{crys}}) ) & \p \mid p,
			\end{cases}
		\end{equation*}
		where $B_{\mathrm{crys}}$ is one of Fontaine's period rings, see e.g. ~Bloch-Kato \cite{BlochKato}. 
	\end{definition}
	While we will not focus on $p$-adic Galois representations, we will use the structure of this definition to guide our definitions of Selmer groups for mod $p$ Galois representations. In particular, we will take the unramified condition at all places not over the characteristic $p$. 

	It is also standard to impose no conditions at the infinite places, which we will do throughout. That is, for $\p$ infinite, we define $L_{\p} = H^1(D_{\p},V)$, so that every class in $H^1(G_K,V)$ automatically satisfies the condition. 
	
	The main object of our interest will be the \textit{nearly-ordinary} Selmer group. 
	\begin{definition}
		The \demph{nearly-ordinary condition} is the local condition 
		\begin{equation*}
			H^1_{\no}(D_{\p},V) \coloneqq \ker\left( H^1(D_{\p},V) \to H^1(I_{\p},V/\ell) \right). 
		\end{equation*}
	\end{definition}

	\begin{definition}
		The \demph{relaxed}, \demph{nearly-ordinary} and \demph{unramified Selmer systems} are defined by 
		\begin{equation*}
			\L_{*} = \begin{cases}
				H^1_{\unr}(D_{\p},V) & \p \nmid p\\ 
				H^1_{*}(D_{\p},V) & \p \mid p .
			\end{cases}
		\end{equation*}
		for $* \in \{ \rel, \no, \unr \}$, with $H^1_{\rel} = H^1(D_{\p},V)$. 
	\end{definition}
	It follows that 
	\begin{equation*}
		\Sel_{\unr}(\rho) \subseteq \Sel_{\no}(\rho) \subseteq \Sel_{\rel}(\rho).
	\end{equation*}
	The inclusion $\no \subseteq \rel$ should be obvious, while $\unr \subseteq \no$ may be less so; note that any cocycle class in $H^1(D_{\p},V)$ vanishing in $H^1(I_{\p},V)$ will also vanish in $H^1(I_{\p},V/\ell)$.

	\section{Cohomology to class field theory}\label{SectionCohoToCFT}
	In this section, we will reinterpret $\Sel_{\L}(\rho)$ in terms of certain Galois extensions of $L$. By translating the local conditions defining $\L$ into the Galois setting, we will give an algorithm to compute the rank of $\Sel_{\L}(\rho)$ for each of $\L= \L_{\rel}, \L_{\no}, \L_{\unr}$. To do this, we will associate to a cocycle class in $H^1(G_K,V)$ a certain homomorphism $f \colon G_L \to V$. Then the field $M_f = \overline{L}^{\ker(f)}$ will give a Galois extension of $L$, whose local properties correspond to the local properties of the original class in $H^1(G_K,V)$. 
	
	\subsection{Inflation-restriction}
	The translation into Galois theory is via the \textit{inflation-restriction} sequence in group cohomology. 
	\begin{definition}
		Let $G$ be a group, $N\triangleleft G$, and $M$ a $G$-module. The \demph{inflation-restriction} sequence is the exact sequence of cohomology groups
		\begin{equation}\label{EquationInflationRestriction}
			0 \to H^1(G/N,M^N) \to H^1(G,M) \to H^1(N,M)^{G/N} \to H^2(G/N,M).
		\end{equation}
	\end{definition}
	Setting $G = G_K$, $N = G_L$, and $M = V$ gives the exact sequence 
	\begin{equation*}
		0 \to H^1(\Gal(L/K),V) \to H^1(G_K,V) \to H^1(G_L,V)^{\Gal(L/K)} \to H^2(\Gal(L/K),V).
	\end{equation*}

	\begin{remark}\label{RemarkActionOnCohomology}
		The third term in the sequence in \cref{EquationInflationRestriction} is $H^1(N,M)^{G/N}$. The action of $G/N$ on the cohomology group can be defined at the level of cocycles: if $f \in Z^1(N,M)$ and $[g] \in G/N$, then 
		\begin{equation}\label{EquationCocycleAction}
			([g] \cdot f)(n) = g^{-1} \cdot f(gng^{-1}).
		\end{equation}
		A short calculation shows the resulting cohomology class is independent of the choice of representative of the class $[g]$. With our choices of $G$, $N$ and $M$, we have the cohomology group 
		\begin{equation*}
			H^1(G_L,V) = \Hom(G_L,V),
		\end{equation*}
		since the action of $G_L = \ker(\rho)$ on $V$ is trivial. Further, $G/N \simeq \Gal(L/K)$, so we have the space of homomorphisms $f \colon G_L \to V$ that are invariant with respect to the action of \cref{EquationCocycleAction}, i.e. homomorphisms such that 
		\begin{equation}\label{EquationHomProperty}
			f(n) = \rho(g)^{-1} f(g n g^{-1})
		\end{equation}
		for all $n \in G_L$, $g \in G_K$. 
	\end{remark}
	Note that \cref{EquationHomProperty} is equivalent to 
	\begin{equation*}
		\rho(g) f(n) = f(g n g^{-1}),
	\end{equation*}
	and so we write $\Hom_{\Gal(L/K)}(G_L,V)$ for the space of all such homomorphisms, as in this form they are equivariant with respect to the group actions on $G_L$ and $V$. There is essentially no difference between equivariance and invariance, it is just a matter of which group actions one considers. 
	
	The inflation-restriction sequence therefore becomes 
	\begin{equation*}
		\scalebox{0.98}{$0 \to H^1(\Gal(L/K),V) \to H^1(G_K,V) \to \Hom_{\Gal(L/K)}(G_L,V) \to H^2(\Gal(L/K),V).$}
	\end{equation*}
	We want to use the middle map to write classes in $H^1(G_K,V)$ as homomorphisms, which requires accounting for the finite cohomology groups $H^i(\Gal(L/K),V)$, $i=1,2$. These only depend on the specific subgroup of $\GL_n(\F_q)$ to which $\Gal(L/K)$ is isomorphic, and it is often the case that both of these groups vanish, giving an isomorphism
	\begin{equation}\label{EquationCohomologyToHoms}
		H^1(G_K,V) \simeq \Hom_{\Gal(L/K)}(G_L,V).
	\end{equation}
	We will assume that these groups always vanish, although there are certainly instances where they do not. For example, the subgroup $H = S_3 = \langle \begin{psmallmatrix}
		2 & 0 \\ 0 & 1
	\end{psmallmatrix}, \begin{psmallmatrix}
		1 & 0 \\ 1 & 1
	\end{psmallmatrix} \rangle$ in $\GL_2(\F_3)$ has $\dim_{\F_3} H^2(H,\F_3^2) = 1$. However, for all the groups we consider, both cohomology groups vanish and we really do have the isomorphism in \cref{EquationCohomologyToHoms}.

	\subsection{Homomorphisms to number fields}
	The following theorem gives a correspondence between $\F_q$-lines in $\Hom_{\Gal(L/K)}(G_L,V)$ and certain Galois extensions of $L$.
	\begin{theorem}\label{TheoremLinesToFields}
		Let $V$ be as above. Then non-trivial lines in $\Hom_{\Gal(L/K)}(G_L,V)$ are in bijection with extensions $M/L$ such that 
		\begin{enumerate}
			\item The extension $M/L$ is abelian and Galois, with $\Gal(M/L) \simeq V$ as an additive group.
			\item The extension $M/K$ is Galois.
			\item The action of $\Gal(L/K)$ on $\Gal(M/L)$ is via $\rho$. 
		\end{enumerate}
	\end{theorem}
	
	\begin{proof}
		This proof is almost entirely a matter of bookkeeping. Most of the hard work is in the ``forward direction''---showing that the extension coming from a homomorphism $f$ satisfies the conditions of the theorem. To perform calculations in the absolute Galois group $G_K$, we fix a section $s \colon \Gal(L/K) \to G_K$ of the quotient map $G_K \to G_K/G_L$. As noted in \cref{RemarkActionOnCohomology}, the action on cocycles (and so homomorphisms) is independent of this choice. 
		
		First, we show the ``forward direction''. Let $f \in \Hom_{\Gal(L/K)}(G_L,V)$. For point (1), recall that the associated extension $M_f/L$ is given by 
		\begin{equation*}
			M_f = \overline{L}^{\ker(f)},
		\end{equation*}
		with $\Gal(M_f/L) \simeq G_L/\ker(f) \simeq f(G_L) \leq V$. If $\alpha \in \F_q^{*}$, we have the equality $\ker(f) = \ker(\alpha f)$, so $M_{\alpha f} = M_{f}$, and the whole space $\langle f \rangle \leq \Hom_{\Gal(L/K)}(G_L,V)$ corresponds to the extension $M_{f}/L$. Assume that $\langle f \rangle$ is a 1-dimensional subspace, i.e. a line, so that $f \neq 0$. 
		
		We show that the image is the \textit{whole} of $V$ by exploiting the equivariance of $f$ (\cref{EquationHomProperty}). Taking $g \in \Gal(L/K)$ and $\tau \in G_L$, we have 
		\begin{equation}\label{EquationConjActionRho}
			\rho(s(g)) f(\tau) = f(s(g) \tau s(g)^{-1}).
		\end{equation}
		So $\im(f)$ is a $\Gal(L/K)$-submodule $V$. Since $V$ is irreducible, $\im(f) = 0$ or $\im(f) = V$. But $f \neq 0$ by assumption (the case $f=0$ corresponds to the trivial extension $M_f = L$), so $\Gal(M/L) = \im(f) = V$. 
		
		Next, we show (2), by showing $\ker(f) \triangleleft G_K$: choose $\tau \in \ker(f)$ and $\sigma \in G_K$; then $f(\sigma \tau \sigma^{-1}) = \rho(\sigma) f(\tau) = \rho(\sigma) 0 = 0$. 
		
		To prove (3), we need to show that the conjugation action of $\Gal(L/K)$ on $\Gal(M/L)$ is via $\rho$. Take $g \in \Gal(L/K)$ and $h \in \Gal(M/L)$. Then we have $s(g) \in G_K$ and $s(h) \in G_L$ such that $\rho(s(g)) = g$, $f(s(h)) = h$. In the absolute Galois groups, the conjugation action is given by 
		\begin{equation*}
			g \cdot h = f(s(g) s(h) s(g)^{-1})
		\end{equation*}
		which is exactly the action in \cref{EquationConjActionRho}, with $\tau = s(h)$. So, $s(g) \cdot s(h) = \rho(s(g)) h$. 
		
		Now we show the ``backwards direction'', that an extension satisfying properties (1-3) gives a homomorphism $f \in \Hom_{\Gal(L/K)}(G_L,V)$.
		
		We have an isomorphism $\Gal(L/K) \simeq G \leq \GL_n(\F_q)$, and an isomorphism $\Gal(M/L) \simeq V = \F_q^n$. The group $G$ acts on $V$ via the standard action of $\GL_n(\F_q)$, and so we can form the semi-direct product $E = V \rtimes G$. The tower $M/L/K$ then satisfies the conditions of the theorem if and only if the exact sequence defining $E$ and the standard exact sequence of Galois groups are isomorphic, i.e. if the following diagram commutes: 
		\begin{equation*}
			\begin{tikzcd}
				1 & {\Gal(M/L)} & {\Gal(M/K)} & {\Gal(L/K)} & 1 \\
				1 & V & E & G & 1
				\arrow[from=1-1, to=1-2]
				\arrow[from=1-2, to=1-3]
				\arrow["\rotatebox{90}{\(\sim\)}", from=1-2, to=2-2]
				\arrow[from=1-3, to=1-4]
				\arrow["\rotatebox{90}{\(\sim\)}", from=1-3, to=2-3]
				\arrow[from=1-4, to=1-5]
				\arrow["\rotatebox{90}{\(\sim\)}", from=1-4, to=2-4]
				\arrow[from=2-1, to=2-2]
				\arrow[from=2-2, to=2-3]
				\arrow["\pi",from=2-3, to=2-4]
				\arrow[from=2-4, to=2-5]
			\end{tikzcd}
		\end{equation*}
		The multiplication in $E$ is defined by the semi-direct product rule: 
		\begin{equation}\label{EquationSemiDirectProduct}
			\pi(e) \cdot v = e v e^{-1}
		\end{equation}
		for $e \in E$ and $v \in V$. We have the quotient maps 
		\begin{align*}
			G_L & \to G_L / G_M \simeq \Gal(M/L), \\ 
			G_K & \to G_K / G_M \simeq \Gal(M/K), \\ 
			G_K & \to G_K / G_L \simeq \Gal(L/K),
		\end{align*}
		which we can add to the above to get the diagram 
		\begin{equation*}
			\begin{tikzcd}
				& {G_L} & {G_K} \\
				1 & {\Gal(M/L)} & {\Gal(M/K)} & {\Gal(L/K)} & 1 \\
				1 & V & E & G & 1
				\arrow[hook, from=1-2, to=1-3]
				\arrow[two heads, from=1-2, to=2-2]
				\arrow[two heads, from=1-3, to=2-3]
				\arrow["r",two heads, from=1-3, to=2-4]
				\arrow[from=2-1, to=2-2]
				\arrow[from=2-2, to=2-3]
				\arrow["\rotatebox{90}{\(\sim\)}", from=2-2, to=3-2]
				\arrow[from=2-3, to=2-4]
				\arrow["\rotatebox{90}{\(\sim\)}", from=2-3, to=3-3]
				\arrow[from=2-4, to=2-5]
				\arrow["\rotatebox{90}{\(\sim\)}", from=2-4, to=3-4]
				\arrow[from=3-1, to=3-2]
				\arrow[from=3-2, to=3-3]
				\arrow["\pi",from=3-3, to=3-4]
				\arrow[from=3-4, to=3-5]
			\end{tikzcd}
		\end{equation*}
		in which all subdiagrams commute. Note that our representation $\rho$ is the composition of $r$ with the isomorphism $\Gal(L/K) \to G$. Define $s \colon \Gal(L/K) \to G_K$, a section of $r$, and define $f$ as the composition of the vertical maps $G_L \to V$. This is certainly a homomorphism, so we only need to check the $\Gal(L/K)$-equivariance. That is, for $g \in \Gal(L/K)$ and $\tau \in G_L$, we must show that we have $f(g \cdot \tau) = g \cdot f(\tau)$. 
		
		First we prove that this is independent of the section $s$: choosing a different lift of $g$ amounts to the difference by an element of $\ker(r) = G_L$, so take a different lift $\sigma s(g)$ for some $\sigma \in G_L$; then
		\begin{align*}
			f(\sigma s(g) \tau s(g)^{-1} \sigma^{-1}) & = f(\sigma) + f(s(g) \tau s(g)^{-1}) + f(\sigma^{-1}) \\ 
			& = f(s(g) \tau s(g)^{-1}).
		\end{align*}
		Meanwhile, the outer action $g \cdot f(\tau)$ is via $\rho$: $g \cdot f(\tau) = \rho(s(g)) f(\tau)$. This is independent of the section $s$, since $\ker(r) = \ker(\rho) = G_L$. 
		
		To show the two actions are equal, examine $s(g) \tau s(g)^{-1}$ in $G_K$, and denote the composition of the vertical maps $G_K \to E$ by $x \mapsto \overline{x}$. Then $\overline{s(g) \tau s(g)^{-1}} = \overline{s(g)} \overline{\tau} \overline{s(g)^{-1}}$, since this is a homomorphism. Now we are in $E$, where the multiplication is as above in \cref{EquationSemiDirectProduct}, so $\overline{s(g)} \overline{\tau} \overline{s(g)^{-1}} = \pi(\overline{s(g)}) \cdot \overline{\tau}$. By the commutativity of the diagram, $\overline{\tau} = f(\tau)$ and $\pi(\overline{s(g)}) = \rho(s(g))$, and we are done. 
	\end{proof}
	
	\subsection{Translating local conditions}
	For each of the three local conditions $\rel$, $\no$ and $\unr$, we will find a corresponding property \textbf{X} of an extension $M/L$, such that $f \in L_{\p}$ if and only if $M_f/L$ has property \textbf{X} at $\p$.

	\subsubsection{The relaxed condition}
	The easiest of these is the relaxed condition, as its contribution to the Selmer group is non-existent: it imposes no condition on cocycle classes in the local cohomology group $H^1(D_{\p},V)$, and so imposes no condition on the corresponding homomorphism given by \cref{EquationCohomologyToHoms}, and so imposes no condition on the corresponding extension $M/L$, beyond those already stated in the theorem. 
	
	\subsubsection{The unramified condition}
	Just as in \cref{EquationInfiniteToFiniteDecomp,EquationInfiniteToFiniteInertia}, we can write the decomposition and inertia groups of $M/L$ as 
	\begin{equation}\label{EquationInfiniteToFiniteDecompInertiaL}
		D_{\q}(M/L) = f(D_{\q}), \spc I_{\q}(M/L) = f(I_{\q}),
	\end{equation}
	where $D_q = \Gal(\Kbar_{\mf{P}}/L_{\q})$. Now, in order to translate the unramified condition, we require a technical lemma describing the inertia subgroup of $G_K$ at a place $\q$ of $L$.
	
	\begin{proposition}
		We have $I_{\q} = I_{\p} \cap G_L$. 
	\end{proposition}
	\begin{proof}
		From the definition of $I_q$ and the inclusion $\Gal(\Kbar_{\mf{P}}/L_{\q}) \hookrightarrow G_L$, it should be clear that $I_{\q} \leq G_L$. From the definitions of $I_{\p}$ and $I_{\q}$, we construct the diagram 
		\begin{equation*}
		\begin{tikzcd}
		\Gal(\Kbar_{\mathfrak{P}} / L_\q ) \arrow[r, two heads] \arrow[d, hook] & \Gal(\F_{\mf{P}}/\F_\q) \arrow[d, hook] \\
		\Gal(\Kbar_{\mathfrak{P}} / K_\p ) \arrow[r, two heads]                 & \Gal(\F_{\mf{P}}/\F_\p)                
		\end{tikzcd}
		\end{equation*}
		of groups, with $I_{\q}$ and $I_{\p}$ the kernels of the horizontal maps. The upper map and lower maps are given by $\sigma \mapsto ([x]  \mapsto [\sigma(x)])$, where $[x]$ denotes the class of $x \in \OO_{\mf{P}}$ in $\F_{\mf{P}}$. Both vertical maps are just restriction, and so the diagram commutes. By commutativity, the kernel of the upper map is in the kernel of the lower, i.e. $I_{\q} \leq I_{\p}$. Then $I_{\q} \leq I_{\p} \cap G_L$. 
		
		For the reverse inclusion, note that if $\sigma \in I_{\p}$, it is the identity in $\Gal(\F_{\mf{P}}/\F_\p)$. If it is in $G_L$ as well, it must fix $L$, and so $L_{\q}$ and $\F_{\q}$ also. So upon restriction, $\sigma$ actually gives an element of the smaller group $\Gal(\Kbar_{\mf{P}} / L_{\q})$. Therefore, $\sigma \in I_{\q}$. 
	\end{proof} 
	Now we can prove the following. 
	\begin{proposition}\label{PropositionTranslatingUnramified}
		A cocycle class $f \in H^1(G_K,V)$ satisfies the unramified local condition $H^1_{\unr}(D_{\p},V)$ if and only if the extension $M_f/L$ is unramified at all places $\q' \mid \p$ of $L$. 
	\end{proposition}
	\begin{proof}
		Let $f \in \Hom_{\Gal(L/K)}(G_L,V)$ satisfy the local condition $H^1_{\unr}(D_{\p},V)$. By this we mean that the cocycle class $g \in H^1(G_K,V)$ that maps to $f$ under the isomorphism in \cref{EquationCohomologyToHoms} satisfies $H^1_{\unr}(D_{\p},V)$, i.e. the restriction of $g$ to $I_{\p}$ is trivial. Applying inflation-restriction with $G = I_{\p}$, $N = I_{\q'}$ and $M=V$, we obtain the commutative diagram of spaces 
		\begin{equation*}
			\begin{tikzcd}
				{H^1(G_K,V)} \arrow[d, "\mathrm{res}"] \arrow[r, "\sim"] & {\Hom_{\Gal(L/K)}(G_L,V)} \arrow[d, "\mathrm{res}"] \\
				{H^1(I_\p,V)} \arrow[r,"\sim"]                   & {\Hom_{I_\p^{\q'}(L/K)}(I_{\q'},V)}         
			\end{tikzcd}
		\end{equation*}
		So $f$ satisfies the unramified local condition if and only if it is trivial in the space $\Hom_{I_\p^{\q'}(L/K)}(I_{\q'},V)$, i.e. if $I_{\q'} \leq \ker(f)$. Then $M/L$ being unramified at $\q'$ follows immediately from \cref{EquationInfiniteToFiniteDecompInertiaL}:
		\begin{equation*}
			I_{\q'}(M/L) = f(I_{\q'}) = 0.
		\end{equation*}
		In fact, the converse follows immediately from this as well. 
	\end{proof}
		
	\subsubsection{The nearly-ordinary condition}
	Finally, to translate the nearly-ordinary condition, consider a line $\ell \leq V$ fixed by $D_{\p}(L/K)$. When one is actually computing this fixed line, it is necessary to choose a particular (finite) decomposition group in $\Gal(L/K)$, which is done by fixing a place $\q \mid \p$. A different choice of decomposition group will give a different fixed line. Call these two choices $D_1$ and $D_2$, fixing $\ell_1$ and $\ell_2$ respectively. Then $D_1 = g^{-1} D_2 g$ for some $g \in \Gal(L/K)$. One immediately sees that $\ell_1 = g(\ell_2)$. 
	
	We will see that this allows us to pick a finite decomposition group at $\p$ in $\Gal(L/K)$ and work entirely with respect to that choice, which will not affect the rank given by the computation of $\Sel_{\no}(\rho)$. 
	
	\begin{proposition}
		A cocycle class $f \in H^1(G_K,V)$ satisfies the local condition $H^1_{\no}(D_\p,V)$ if and only if $I_{\q}(M/L) \leq \ell$, where $\q \mid \p$ defines the line $\ell$. 
	\end{proposition}
	\begin{remark}
		Here, \textit{a priori}, $I_{\q}(M/L)$ is only defined in $\Gal(M/L)$ up to conjugacy. But since $\Gal(M/L)$ is abelian, there is no ambiguity. 
	\end{remark}
	\begin{proof}
		We can add another row to the commutative diagram in the proof of \cref{PropositionTranslatingUnramified} by taking the quotient of $V$ by the fixed line $\ell$: 
		\begin{equation*}
			\begin{tikzcd}
				{H^1(G_K,V)} \arrow[r, "\sim"] \arrow[d]  & {\Hom_{\Gal(L/K)}(G_L,V)} \arrow[d]          \\
				{H^1(I_\p,V)} \arrow[r, "\sim"] \arrow[d] & {\Hom_{I_\p(L/K)}(I_\q,V)} \arrow[d] \\
				{H^1(I_\p,V/\ell)} \arrow[r, "\sim"]       & {\Hom_{I_\p(L/K)}(I_\q,V/\ell)}      
			\end{tikzcd}
		\end{equation*}
		Assume $f \in \Hom_{\Gal(L/K)}(G_L,V)$ satisfies the nearly-ordinary condition, i.e. its image in $\Hom_{I_\p(L/K)}(I_\q,V/\ell)$ is trivial. Then $f(I_{\q}) \leq \ell$. Again we use that $f(I_{\q}) = I_{\q}(M/L)$ to see that, therefore, $I_{\q}(M/L) \leq \ell$. And, as in the proof of \cref{PropositionTranslatingUnramified}, this also proves the converse, that $M/L$ with $I_{\q}(M/L) \leq \ell$ satisfies the nearly-ordinary condition. 
	\end{proof}

	\subsection{Class field theory}
	By assumption, $V \simeq \F_p \oplus \F_p$ as an additive group, so an extension $M/L$ with $\Gal(M/L)\simeq V$ will be a subextension of the maximal abelian extension of $L$. Further, as we take the unramified condition for all primes of $K$ not over $p$ (which translates to the extension $M/L$ being unramified), $M/L$ in fact lies in the maximal abelian extension of $L$ unramified outside of $x \mid p$. As an immediate consequence we have the following: 
	\begin{proposition}
		Selmer groups of mod $p$ Galois representations are finite. 
	\end{proposition}
	This follows from the classical result that there are only finitely many extensions of a number field $L$ of a fixed degree and finite set of ramifying primes. 
	
	This maximal extension can be computed using class field theory. There is some modulus $\mf{m}$ such that the ray class field $L(\mf{m})$ is the maximal abelian extension of $L$, unramified outside $\mf{m}$. Subfields corresponding to elements in the relaxed Selmer group can then be found using the Galois theory of $L(\mf{m})/L$, and further refined to the nearly-ordinary and unramified Selmer groups by studying their ramification properties.

	\section{Technical considerations}\label{SectionTechnical}
	In this section, we note some important technical details used during implementation of the algorithm. 
	
	\subsection{Choosing the modulus}
	To find the maximal abelian extension of $L$ unramified away from $p$, we need to include all primes $\p \mid p$ in our modulus. Let $M/L$ be an extension corresponding to a line in $\Sel_{\rel}(\rho)$, realised in class field theory by the conductor $\mf{f}$. Exercise 6 of Appendix A of Cohen \cite{CohenAdvancedTopics} tells us that 
	\begin{equation}\label{EquationMaxExp}
		v_\p(\mf{f}) \leq \left\lfloor \frac{2pe}{p-1} \right\rfloor + 1
	\end{equation}
	where $e = e(\p/p)$ is the ramification index of a prime $\p \mid p$ in $L$. Then $M/L$ is contained in the field ray class field of the modulus 
	\begin{equation*}
		\mf{m} = \prod_{ \p \mid p } \p^{\left\lfloor \frac{2pe}{p-1} \right\rfloor + 1}.
	\end{equation*}
	We allow any of the infinite places of $L$ to ramify.
	
	In practice, the modulus $\mf{m}$ is often larger than we require. Once we have cut out the maximal $p$-extension $A/L$ inside the ray class field of $\mf{m}$, we can replace it with the ray class field defined by the modulus equal to the conductor of $A$. In practice, this eliminates unnecessary information in the ray class field, and speeds up the computation. 
	
	\subsection{Compatibility of the choice of fixed line}\label{SectionCompatibleLineChoice}
	When performing actual calculations with nearly-ordinary Selmer groups, it is necessary to make a choice of decomposition group in $\Gal(L/K)$, which arises from a choice of prime $\q \mid \p$ in $L$. This choice $\q$ fixes the decomposition group $D_{\p}^{\q}(L/K)$, and, further, the line $\ell$ that $\rho$ fixes when restricted to this subgroup. We need to understand how the result of the computation of $\dim \Sel_{\no}(\rho)$ depends on this choice. 
	
	Let $\q' \mid \p$ in $L$. Since $\Gal(L/K)$ acts transitively on the primes of $L$ over $\p$, there is some $g \in \Gal(L/K)$ such that $\q' = g(\q)$. Further, note that 
	\begin{equation*}
		g D_{\p}^{\q}(L/K) g^{-1} = D_{\p}^{\q'}(L/K)
	\end{equation*}
	and that $D_{\p}^{\q'}(L/K)$ fixes the line $g(\ell) \leq V$. 
	\begin{proposition}
		With the above setup, we have 
		\begin{equation*}
			I_{\q}(M/L) \leq \ell \iff I_{\q'}(M/L) \leq g(\ell).
		\end{equation*}
	\end{proposition}
	\begin{proof}
		This follows immediately from the equality of inertia groups $g \cdot I_{\q}(M/L) = I_{\q'}(M/L)$, which we now show. Take $h \in I_{\q}(M/L)$ and $x \in \mathcal{O}_{M/L}$. Recall that $g$ acts on $\Gal(M/L)$ (and so the subgroup $I_{\q}(M/L)$) by conjugation by an extension $\tilde{g}$ of $g$ to $\Gal(M/K)$. Writing $y=\tilde{g}^{-1}(x)$, we have 
		\begin{align*}
			(g \cdot h)(x) - x & = (\tilde{g} h \tilde{g}^{-1})(x) - x  \\ & = (\tilde{g}h)(y) - \tilde{g}(y) \\ 
			& = \tilde{g}(y+q) - \tilde{g}(y) \\ 
			& = \tilde{g}(q),
		\end{align*}
		where $q \in \tilde{\q}$, for $\tilde{\q}$ a choice of place dividing $\q$ in $M$. \textit{A priori} the inertia group $I_{\q}(M/L)$ depends on this choice, and is only defined up to conjugacy in $M/L$. But $\Gal(M/L)$ is abelian, so the inertia subgroup is independent of $\tilde{\q}$. Taking some place $\tilde{\q'}$ in $M/L$ over $\q'$, we note that $\tilde{\sigma}(\tilde{\q}) = \tilde{\q'}$, and so $(g \cdot h)(x) -x \in \tilde{\q'}$, i.e. $g \cdot h \in I_{\q'}(M/L)$. 
	\end{proof}
	From this, we know the rank is independent of the choice of $\q$, since a different choice $\q'$ yields a different fixed line, and either both inertia groups are contained in their respective fixed lines, or neither is. 
	
	\subsection{\textsf{Magma}'s \texttt{FldNum} versus \texttt{FldAb}}
	In \cref{SectionExamples}, we describe the extensions $M/L$ and their ramification properties explicitly. In the actual algorithm, we make heavy use of \textsf{Magma}'s \texttt{FldAb} type, where computations of these kinds are much faster, as they use Fieker's algorithms with the Artin map \cite{FiekerCFT}. For $\rho$ a representation with image $\GL_2(\F_2)$, the extensions $M/K$ will be of degree $\abs{\F_2^2} \cdot \abs{\GL_2(\F_2)} = 24$, which is in the range of manageable computations for \texttt{FldNum} types. But for the $\SD_{16} \leq \GL_2(\F_3)$ examples we consider, the degree of $M/K$ becomes $144$. With the current algorithm, these examples cannot be computed with as \texttt{FldNum}s in any reasonable amount of time. 
	
	The main improvement offered by the \texttt{FldAb} type is fast computation of inertia groups. Suppose we have $M/L$ and want to compute the inertia group at a prime $\p$ of $L$. The field $M$ is given as a subfield of $L(\mf{m})$ for an appropriately-chosen modulus $\mf{m}$, and the modulus $\mf{m}' = \mf{m}/\p^{v_\p(\mf{m})}$ defines the maximal abelian extension of $L$ unramified outside of $p$ \textit{and} the ideal $\p$. Then 
	\begin{equation*}
		M^{I_{\p}(M/L)} = M \cap L(\mf{m}').
	\end{equation*}
	This is a computation \textsf{Magma} performs quickly, and allows us to quickly infer the size of $I_{\p}(M/L)$. \textsf{Magma} also allows for the computation of $M^H$ for any $H \leq \Gal(M/L)$ in essentially the same way. Then, for a fixed line $\ell$ defining the nearly-ordinary Selmer group, we have $I_{\p}(M/L) \leq \ell$ if and only if $M^{\ell} \subset M^{I_{\p}(M/L)}$. 
	
	\subsection{Class group bounds}
	All of our computations use the class group bounds afforded by assuming the Generalised Riemann Hypothesis, which makes many examples tractable.

	\section{Examples}\label{SectionExamples}
	We illustrate the method with some examples. Some details, such as exact identifications of field automorphisms and elements, might seem extraneous, but we provide them in the interest of demonstrating that the method is completely explicit. 
	
	\begin{remark}
		During the rewriting of the code for this paper, we found some subtle bugs that affected some computations over imaginary quadratic fields. These bugs have been fixed, and the updated data is disseminated alongside the code at the above Github repository. 
	\end{remark}
	
	\subsection{Nearly-ordinary $\boldsymbol{\GL_2(\F_2)}$}
	Let 
	\begin{equation*}
		E \colon y^2 = x^3 - x^2 - x + 2
	\end{equation*}
	be an elliptic curve; it has conductor $236$, and LMFDB label \LMRefDB{https://www.lmfdb.org/EllipticCurve/Q/236/a/1}{236.a1}. Its $2$-torsion field is given by $L = \Q(\alpha)$, with $\alpha$ a root of  
	\begin{equation*}
		x^6 - 3x^5 + 10x^4 - 15x^3 + 21x^2 - 14x + 4=0.
	\end{equation*}
	The automorphism group of $L/\Q$ is $S_3 \simeq \GL_2(\F_2)$; it is generated by 
	\begin{align*}
		\sigma \colon \alpha & \mapsto \frac{1}{2}(\alpha^5 - 3\alpha^4 + 10\alpha^3 - 13\alpha^2 + 19\alpha - 6), \\ 
		\tau \colon \alpha & \mapsto \frac{1}{2}(-\alpha^5 + 3\alpha^4 - 10\alpha^3 + 13\alpha^2 - 19\alpha + 8),
	\end{align*}
	with $\sigma^3 = \tau^2 = 1$, $\tau \sigma \tau = \sigma^2$. We pick the mod $2$ representation $\rho \colon \Gal(L/\Q) \to \GL_2(\F_2)$ given by 
	\begin{equation*}
		\sigma \mapsto \begin{pmatrix}
			0 & 1 \\ 1 & 1
		\end{pmatrix}, \spc \tau \mapsto \begin{pmatrix}
			1 & 1 \\ 0 & 1
		\end{pmatrix}.
	\end{equation*}	
	We have $2\Z_L = \p_1 \p_2 \p_3$, where 
	\begin{align*}
		\p_1 & = (2, 2+\alpha^3), \\
		\p_2 & = (2, 1 + \alpha + \alpha^2), \\
		\p_3 & = (2, 3+\alpha+\alpha^2+\alpha^3).
	\end{align*}
	For each $\p_i$, the maximal exponent from \cref{EquationMaxExp} is $5$, so we take the modulus $\mf{m}  = (2\Z_L)^5$. The associated abelian extension $L(\mf{m})$ of $L$ has Galois group 
	\begin{equation*}
		\Gal(L(\mf{m})/L) \simeq (\Z/2)^2 \oplus (\Z/4)^3 \oplus \Z/24,
	\end{equation*}
	and so the maximal abelian $2$-extension $A$, unramified outside of $2$, has $\Gal(A/L) \simeq (\Z/2)^6$. The action of $\Gal(L/\Q)$ on $\Gal(A/L)$ is given by 
	\begin{equation*}
		\sigma \mapsto \begin{pmatrix}
			1 & 1 & 0 & 0 & 1 & 0 \\ 
			1 & 0 & 0 & 0 & 0 & 1 \\ 
			1 & 0 & 1 & 0 & 1 & 1 \\
			1 & 0 & 0 & 1 & 1 & 1 \\
			0 & 0 & 0 & 0 & 0 & 1 \\
			0 & 0 & 0 & 0 & 1 & 1 
		\end{pmatrix},\spc \tau \mapsto \begin{pmatrix}
			0 & 1 & 0 & 0 & 1 & 0 \\ 
			1 & 0 & 0 & 0 & 1 & 0 \\ 
			0 & 0 & 1 & 0 & 0 & 0 \\ 
			0 & 0 & 0 & 1 & 0 & 0 \\ 
			0 & 0 & 0 & 0 & 1 & 0 \\ 
			0 & 0 & 0 & 0 & 1 & 1
		\end{pmatrix}.
	\end{equation*}
	Extensions $M/L$ satisfying the conditions of \cref{TheoremLinesToFields} are in bijection with $\Gal(L/K)$-submodules\footnote{Submodules rather than simply subgroups, as $\Gal(L/K)$ acts on $\Gal(M/L)$.} of $\Gal(A/L)$ that are isomorphic to $(\Z/2)^2$, of which there are $4$. Thus, we find $4$ normal subfields $M/L$ with $\Gal(M/L) \simeq V$; these are given by 
	\begin{align*}
		M_1 & = L(\sqrt{\beta_1},\sqrt{\beta_2}), \\
		M_2 & = L(\sqrt{\beta_3},\sqrt{\beta_4}), \\ 
		M_3 & = L(\sqrt{\beta_1 \beta_3},\sqrt{\beta_2 \beta_4}), \\ 
		M_4 & = L(\sqrt{2},\sqrt{-1}),
	\end{align*}
	where 
	\begin{align*}
		\beta_1 & = -\alpha^3 + \alpha^2 - 3\alpha, \\ 
		\beta_2 & = \alpha^2 - \alpha + 2, \\ 
		\beta_3 & = \frac{1}{2}(-\alpha^4 + 2\alpha^3 - 4\alpha^2 + \alpha),  \\ 
		\beta_4 & = \frac{1}{2}(-\alpha^4 + 2\alpha^3 - 4\alpha^2 + 3\alpha),\\ 
	\end{align*}
	are elements of $\Z_L$ with norms $16$, $16$, $1$ and $1$ respectively. 
	
	First, to compute the relaxed Selmer group, we find the subset of the $M_i$ acted upon by $\Gal(L/\Q)$ via $\rho$; this turns out to be $\{M_1,M_2,M_3\}$ (the action on $M_4$ is trivial), giving three lines in $\Sel_{\rel}(\rho)$, i.e. 
	\begin{equation*}
		\rank(\Sel_{\rel}(\rho)) = 2.
	\end{equation*}
	
	For the nearly-ordinary group, we choose the prime $\p_1$ over $2$ to compute the decomposition group 
	\begin{equation*}
		D_2^{\p_1}(L/K) = \langle \sigma \tau \rangle = \langle \begin{psmallmatrix}
			0 & 1 \\ 1 & 0
		\end{psmallmatrix} \rangle,
	\end{equation*}
	which fixes the line $\ell = \{ 0, \begin{psmallmatrix}
		1 \\ 1
	\end{psmallmatrix} \}$. In $M_2$ and $M_3$, the prime $\p_1$ factors as $\q_1^4$, so we immediately conclude that $\abs{I_{\p_1}(M_i/L)} = 4$ for $i=2,3$, and so we cannot have $I_{\p_1}\leq \ell$ for these fields. For $M=M_1$, we find that the inertia group of $\p_1$ is generated by the map 
	\begin{equation*}
		\mu \colon \begin{cases}
			\beta_1 & \mapsto -\beta_1 \\ 
			\beta_2 & \mapsto -\beta_2,
		\end{cases}
	\end{equation*}
	This subgroup is exactly the line $\ell$, so we have $I_{\p_1}(M/L)\leq \ell$, and so 
	\begin{equation*}
		\rank(\Sel_{\no}(\rho)) = 1.
	\end{equation*}
	
	Finally, since none of the $M_i$ for $i=1,2,3$ have trivial inertia group for $\p_1$ (equivalently, any of the primes over $2$ in $L$), we have that 
	\begin{equation*}
		\rank(\Sel_{\unr}(\rho)) = 0.
	\end{equation*}
	
	In \cref{SectionStatsGL22}, we note that all the nearly-ordinary mod $2$ Selmer groups we compute have positive rank. In this instance, the extension realising this rank is given by $M_1$, generated by the square roots of two elements of norm $16$. 
	
	\subsection{Multiple fixed lines from $\boldsymbol{\SD_{16} \leq \GL_2(\F_3)}$}\label{SectionSD16}
	The group 
	\begin{equation*}
		\SD_{16} = \langle a,b \mid a^2 = b^8 = 1, a^{-1}ba = b^3 \rangle
	\end{equation*}
	is a non-abelian group of order $16$, lying inside $\GL_2(\F_3)$. Let 
	\begin{align*}
		E \colon y^2=x^3+x^2-9x+55
	\end{align*}
	be the elliptic curve with LMFDB label \LMRefDB{https://www.lmfdb.org/EllipticCurve/Q/3136/d/2}{3136.d2}. Its mod $3$ representation is cut out by a field $L$ with $\Gal(L/\Q) \simeq \SD_{16}$. We choose the representation such that 
	\begin{equation*}
		\rho(a) = \begin{pmatrix}
			1 & 0 \\ 2 & 2
		\end{pmatrix}, \spc \rho(b) = \begin{pmatrix}
			0 & 2 \\ 2 & 1
		\end{pmatrix}.
	\end{equation*}
	We compute that $D_3(L/K) = \langle \begin{psmallmatrix}
		2 & 0 \\ 1 & 1
	\end{psmallmatrix} \rangle \simeq C_2$, which fixes lines $\ell_1 = \{0, \begin{psmallmatrix}
		0 \\ 1
	\end{psmallmatrix}\}$, $\ell_2 = \{ 0, \begin{psmallmatrix}
		1 \\ 1
	\end{psmallmatrix} \} $ in $\F_3^2$. The action of $\rho$ on $V/\ell_i$ is, respectively, ramified for $i=1$ and unramified for $i=2$. Both lines give rise to their own local condition at $3$, and their own Selmer groups. We compute the ranks 
	\begin{equation*}
		\rank(\Sel_{\rel}(\rho)) = 2,
	\end{equation*}
	\begin{equation*}
		\rank(\Sel_{\no,\ell_1}(\rho)) =1, \spc \rank(\Sel_{\no,\ell_2}(\rho)) =2,
	\end{equation*}
	and $\rank(\Sel_{\unr}(\rho)) = 1$. This example shows that the Selmer groups arising from different fixed lines in the same representation need not have the same dimension. This discrepancy in ranks is investigated further in \cref{SectionSD16Stats}. We will also use the nearly-ordinary Selmer group with unramified quotient in \cref{SectionBKTest}.

	\subsection{Statistics}
	In this subsection we give some statistics for these Selmer groups for various groups and fields. 
	
	\subsubsection{Image $\GL_2(\F_2)$}\label{SectionStatsGL22}
	We compute the relaxed, nearly-ordinary and unramified ranks for Selmer groups associated to representations with image $\GL_2(\F_2)$ coming from elliptic curves up to various conductor bounds listed in \cref{TableGL22Ranks}. We exclude all elliptic curves with conductor divisible by a prime over $2$, and all those whose mod $2$ representation is not nearly-ordinary. In all cases, we use LMFDB data on elliptic curves, which are complete up to the range listed in the table. Acknowledgments for the sources of these data can be found on the LMFDB itself. In all cases, the rank of the nearly-ordinary Selmer group is positive. 
	\begin{figure}
		\renewcommand{\arraystretch}{1.5}
		\begin{tabular}{|c|c|c|c|c|c|}
			\hline 
			\raisebox{-0.6ex}[0pt][0pt]{Base} & \raisebox{-0.6ex}[0pt][0pt]{Conductor} & \raisebox{-0.6ex}[0pt][0pt]{Number of} & \multicolumn{3}{c|}{Mod $p$ Selmer rank}\\ \cline{4-6} \raisebox{0.6ex}[0pt][0pt]{field} & \raisebox{0.6ex}[0pt][0pt]{bound} & \raisebox{0.6ex}[0pt][0pt]{NO reps} & $\rel$ & $\no$ & $\unr$ \\ \hline 
			$\Q$ & $500{,}000$ & $38{,}497$ & $2.49$ & $1.52$ & $0.41$  \\ \hline 
			$\Q(\sqrt{-3})$ & $150{,}000$ & $2{,}373$ & $3.05$ & $1.46$ & $0.07$  \\ \hline 
			$\Q(\sqrt{-1})$ & $100{,}000$ & $1{,}055$ & $3.03$ & $1.92 $ & $0.06$  \\ \hline 
			$\Q(\sqrt{-7})$ & $50{,}000$ & $76$ & $4.12$ & $2.04$ & $0.08$  \\ \hline 
			$\Q(\sqrt{-2})$ & $50{,}000$ & $848$ & $3.04$ & $1.55$ & $0.06$  \\ \hline 
			$\Q(\sqrt{-11})$ & $50{,}000$ & $1{,}211$ & $3.06$ & $1.44$ & $0.08$ \\  \hline 
		\end{tabular}
		\caption{Average ranks of the relaxed, nearly-ordinary and unramified mod $2$ Selmer groups over various number fields.}\label{TableGL22Ranks}
	\end{figure}
	
	\subsubsection{Image $\SD_{16}$}\label{SectionSD16Stats}
	We compute the relaxed, nearly-ordinary and unramified ranks for Selmer groups associated to representations with image $\SD_{16}$, coming form elliptic curves over $\Q$, again with conductor up to $500{,}000$. In the table in \cref{TableSD16Ranks}, we collect data on the averages of these ranks. In all cases, the nearly-ordinary $\SD_{16}$ representation fixed two lines, leading to a ramified and unramified quotient. Per the table, the data indicate that the rank associated to the unramified quotient seems to be larger, on average, than that associated to the ramified quotient. 
	
	\begin{figure}
		\renewcommand{\arraystretch}{1.5}
		\begin{tabular}{|c|c|c|c|c|}
			\hline 
			Selmer group & rel & NO (unr.) & NO (ram.) & unr \\ \hline 
			Average rank & $1.597$ & $0.946$ & $0.763$ & $0.231$ \\ \hline 
		\end{tabular}
		\caption{Average ranks of the relaxed and unramified Selmer groups of nearly-ordinary $\SD_{16}$ representations over $\Q$, and ranks of the nearly-ordinary Selmer groups associated to ramified and unramified quotients.}\label{TableSD16Ranks}
	\end{figure}

	\section{Testing a Bloch-Kato-type relationship}\label{SectionBKTest}
	In this section, we examine how closely the mod $p$ nearly-ordinary Selmer group mimics an important property of its $p$-adic cousin, namely its relation to the Bloch-Kato Selmer group. As noted in the introduction, no ``mod $p$ Bloch-Kato Selmer group'' has yet been proposed, so we must make the comparison indirectly, using periods of mod $p$ cohomology classes. We do this using the $\SD_{16}$ representations computed in the previous section.  
	
	\subsection{Serre's conjecture}
	Serre's conjecture (now a theorem of Khare and Winterberger \cite{KhareWint} over $\Q$) relates mod $p$ Galois representations to cohomology classes. Specifically, if $\rho$ is an odd, absolutely irreducible, continuous two-dimensional mod $p$ representation of $G_\Q$, then there is an associated modular eigenform $f$ whose mod $p$ Galois representation is isomorphic to $\rho$. The strong form of Serre's conjecture attaches invariants to $\rho$ that determine exactly where to find this eigenform. The Serre conductor $N(\rho)$, a character $\chi$ of $\Z / N(\rho)\Z$, and a weight $k$, such that $f \in S_{k}(N,\chi)$. For more details on these constructions, see Serre \cite{Serre}. For all of the representations we consider here, the weight $k$ is always $2$ and the character $\chi$ is always trivial, which we will assume to simplify the discussion. We also assume that $f$ has real Hecke eigenvalues.
	
	\subsection{$L$-values and periods}
	The important quantity for mod $p$ representations is not the classical $L$-value $L(f,1)/\Omega_f$, but its ``mod $p$ reduction''. While the $L$-value is (conjecturally) transcendental, we can find a number $\Omega_f$, the \textit{least real period of $f$}, such that $L(f,1)/\Omega_f  \in \Q$. This period is determined by scaling the $q$-expansion of $f$ so the coefficient of $q$ is $1$; then the integrals $\int_{\tau}^{\gamma \cdot \tau} f(z)dz$ for all $\gamma \in \Gamma_0(N)$ form a lattice in $\C$, whose smallest real element is $\Omega_f$. 
	
	We can reduce the rational number $L(f,1)/\Omega_f$ modulo $p$ via $\frac{a}{b} \mapsto a b^{-1} \Mod{p}$, with $b^{-1}$ a multiplicative inverse of $b$ modulo $p$. We expect that the vanishing or not of this quantity should be related to the presence (or lack) of rank in a hypothetical mod $p$ Bloch-Kato Selmer group. 
	
	The mod $p$ reduction of $L(f,1)/\Omega_f$ can, in fact, be computed without the need for integration or evaluating an $L$-series. Using the multiple period polynomial approach of Pa\c{s}ol-Popa \cite{PasolPopa} in characteristic $3$, we can compute the $1$-dimensional space of weight $2$ multiple period polynomials of level $\Gamma_0(N(\rho))$ whose Hecke eigenvalues match the traces of Frobenius of $\rho$. Then a certain coefficient of this polynomial (coming from the identity coset of $\Gamma_0(N(\rho))$ in $\SL_2(\Z)$) gives the reduction of $L(f,1)/\Omega_f$ modulo $3$. 
	
	Working directly with mod $3$ period polynomials gives two advantages. First, we do not have to worry about picking out the precise eigenspace in characteristic $0$ reducing to the one we are looking for mod $3$---when the conductor $N(\rho)$ is equal to the conductor of the elliptic curve from which $\rho$ arises, this is easy, as it will be a $1$-dimensional space with integer Hecke eigenvalues, but when $N(\rho)$ is smaller, the associated form will have eigenvalues in some number field. Picking the correct conjugate and an ideal in the number field so that the reduction lines up with $\rho$ is doable, but unnecessary if we just work over $\F_3$ from the start. And second, calculations in characteristic $p$ are generally much quicker than in characteristic $0$, meaning we can compute periods more efficiently. 
	
	Finally, we note that it is not enough to just record the vanishing of $L$-values in characteristic $0$, as the rational number $L(f,1)/\Omega_f$ may have a $3$ in its numerator, so a representation $\rho$ coming from an elliptic curve with rank $0$ (and so non-zero $L$-value, subject to BSD) may still have its period equal to $0 \Mod{3}$. And more generally, for Serre's conjecture over other fields, it is not known that the mod $p$ cohomology class associated to a mod $p$ Galois representation will lift to a class in characteristic $0$. In that setting, the mod $p$ period may be the only robust piece of information one has.	
	
	\subsection{The calculation}
	For each $\rho$ in the dataset from \cref{SectionSD16Stats}, we compute the rank of the nearly-ordinary Selmer group with respect to the fixed line $\ell$ such that the quotient $\rho^*$ is unramified. We also compute the conductor of $\rho$ in order to find the period of its associated mod $p$ cohomology class. This allows us to test the proposed relationship between the rank and the period: that the vanishing of the period implies the non-vanishing of the rank, and vice versa, up to a possible overestimation of the rank by at most $1$. 
	
	Note, when the rank is $1$ we cannot tell if we are in the case of overestimating a rank of $0$ by $1$, where expect a non-zero period, or if there is no overestimation, with expected period $0$. So we are really testing the (slightly fuzzier) relationship 
	\begin{equation}\label{EquationFuzzierRelationship}
		\text{rank} = 0 \implies \text{period} \neq 0, \spc 
		\text{rank} \geq 2 \implies \text{period} = 0. 
	\end{equation}
	
	Of the $186$ representations we computed, $72$ have nearly-ordinary rank (with unramified quotient) not equal to $1$. To compute the periods for these representations, we use a combination of direct computation and LMFDB data. When the conductor of $\rho$ is equal to the conductor of the elliptic curve giving rise to the representation, we can find the rational value $L(E,1)/\Omega_E$ directly using data from the LMFDB. 
	
	In the case that the conductor of $\rho$ is less than that of $E$, there is a mod $3$ modular form with the correct Hecke eigenvalues at this smaller level, but it will be the reduction modulo $3$ of a non-rational modular form. For these periods, we use the extended period polynomials. 
	
	Of the $72$ representations with rank not equal to $1$, only $9$ have conductor smaller than their associated elliptic curves (these are the representations coming from the curves 
	\LMRefDB{https://www.lmfdb.org/EllipticCurve/Q/18605/c/1}{18605.c1}, \LMRefDB{https://www.lmfdb.org/EllipticCurve/Q/94178/bb/1}{94178.bb1}, \LMRefDB{https://www.lmfdb.org/EllipticCurve/Q/153760/d/1}{153760.d1}, \LMRefDB{https://www.lmfdb.org/EllipticCurve/Q/153760/f/1}{153760.f1}, \LMRefDB{https://www.lmfdb.org/EllipticCurve/Q/266450/d/1}{266450.d1}, \LMRefDB{https://www.lmfdb.org/EllipticCurve/Q/266450/i/1}{266450.i1}, \LMRefDB{https://www.lmfdb.org/EllipticCurve/Q/307520/d/1}{307520.d1}, \LMRefDB{https://www.lmfdb.org/EllipticCurve/Q/307520/v/1}{307520.v1}, \LMRefDB{https://www.lmfdb.org/EllipticCurve/Q/307520/w/1}{307520.w1}), whose periods we compute directly. 
	
	For all representations, the relationship \cref{EquationFuzzierRelationship} between ranks and periods is satisfied. 
	
	Of particular interest are cases such as the representation coming from the curve \LMRefDB{https://www.lmfdb.org/EllipticCurve/Q/453152/bq/1}{453152.bq1}, which has mod $3$ Selmer rank $2$, but $\mathrm{rank}(E) = 0$. The $L$-value is $L(E,1)/\Omega_E = 18$, with the factor of $9 = 3^2$ coming from the (analytic) rank of $\Sha(E/\Q)$, perhaps accounting for the rank of the mod $3$ Selmer group.

\end{document}